\documentclass[a4paper,10pt]{amsart}
\usepackage[utf8]{inputenc}
\usepackage{graphicx}
\usepackage{amssymb}
\usepackage{color}
\usepackage{todonotes}
\usepackage{verbatim}
\theoremstyle{plain}
\newtheorem{Theorem}{Theorem}
\newtheorem{Lemma}{Lemma}
\newtheorem{Corollary}{Corollary}
\newtheorem{Proposition}{Proposition}

\theoremstyle{definition}
\newtheorem{Definition}{Definition}
\theoremstyle{remark}
\newtheorem{Remark}{Remark}

\title[Noise Induced Order for Skew Products]{Noise Induced Order for Skew-Products over a Non-Uniformly expanding base}
\author{A. Blumenthal, I. Nisoli}

\newcommand{\pd}{\partial}

\newcommand{\Leb}{\operatorname{Leb}}
\newcommand{\Lip}{\operatorname{Lip}}

\newcommand{\R}{\mathbb R}
\newcommand{\uo}{{\underline{\omega}}}
\newcommand{\ue}{{\underline{\eta}}}
\newcommand{\Fc}{\mathcal F}
\renewcommand{\P}{\mathbb P}

\begin{document}

\begin{abstract}
Noise-induced order is the phenomenon by which the chaotic regime of a 
deterministic system is destroyed in the presence of noise. In this manuscript, 
we establish noise-induced order for a natural class of systems of dimension $\geq 2$ 
consisting of a fiber-contracting skew product a over nonuniformly-expanding 1-dimensional system. 
\end{abstract}

\maketitle

\section{Introduction}
In recent times, interest in random dynamical systems has been greatly stimulated 
due to their use in applications and their scientific relevance 
in modeling systems driven by external or internal sources of noise. 
It is of both theoretical and practical interest to understand
the ways in which dynamical behavior can change under the
influence of noise. Notable examples include 
 stochastic resonance phenomena (e.g., 
 \cite{GammaEtAl, Benzi} or the recent mathematical work \cite{SatoLamb}), where
a stable system can be excited in the presence of noise
into producing oscillatory behavior; and noise-induced
chaos, where a sufficient amount of noise can induce
chaotic behavior in the random dynamics (e.g., 
\cite{Gao, Gao2, Blu}). 


The topic of this manuscript is noise-induced order (NIO), 
referring to scenarios where the presence of noise
induces stabilization in a previously chaotic deterministic system, quantitatively measured through a transition of the top Lyapunov exponent from positive 
to negative as the noise amplitude increases.
This surprising phenomenon was first observed by numerical experiments in a one dimensional 
model of the Belosouv-Zhabotinsky reaction \cite{MatTs}. A mathematical proof of this phenomenon
was given only recently in \cite{GaMoNi} via computer-assistance. The recent paper \cite{Ni} describes a
 sufficient condition for the existence of NIO in one
dimensional non-uniformly expanding systems. 



Our purpose here is to establish sufficient conditions for 
the existence of Noise Induced Order in fiber-contracting 
skew-products over a non-uniformly hyperbolic base dynamics in the presence of additive noise. 
As we show here, our abstract framework applies to a fundamental model in dynamics: the Poincaré map 
of a transverse section for the contracting geometric Lorenz flow, 
which we refer to hereafter as the \emph{contracting Lorenz $2$-dimensional map}.


The classical Lorenz model \cite{Lorenz63, Tucker} is an important example in nonlinear sciences
and a prototypical example of deterministic chaotic behavior. 
A related model, the so-called geometric Lorenz flow \cite{Guck}, was
constructed so as to capture qualitative features of
the Poincar\'e map for the Lorenz model at a natural transversal section. 
Like the classical Lorenz flow, geometric Lorenz flows
admit a saddle equilibrium at the origin, but due to the flexibility of their construction it is possible to parametrically adjust 
features of these models such as the eigenvalues at the origin. 
Contracting Lorenz flow refers to these geometric models when the contracting eigenvalues `dominate' the expanding eigenvalue. 

Contracting Lorenz flow has been extensively studied. 
Metzger and Morales 
proved its stochastic stability \cite{Me, MeMo}, while Alves and Soufi proved 
statistical stability of the associated Poincar\'e maps \cite{AlSou}. Galatolo, Nisoli and Pacifico 
proved that the $2$-dimensional map at a Rovella parameter exhibits 
exponential decay of correlations with respect to Lipschitz observables \cite{GaNiPa}. 
A thermodynamic formalism for contracting Lorenz maps was developed by Pacifico and Todd 
\cite{PaTo}.  Recent works of Alves -- Khan and Araujo proved, respectively,
that the contracting Rovella flow, a perturbed variant of contracting Lorenz flow, is not statistically stable if we consider 
all the perturbations in the $C^3$ topology \cite{AlKhan}, but is statistically stable 
if we consider perturbations inside the family of so-called Rovella parameters \cite{Ar}.

Our result on noise-induced order for the contracting Lorenz map leads us to conjecture that the contracting Lorenz flow itself exhibits noise-induced order. 
The proof of this result, if true, would likely require computer-assisted tools. 


\subsubsection*{Plan for the paper}
In Section \ref{sec:setting}
we formulate an abstract framework of fiber-contracting skew products and formally state our result on Noise-Induced Order. The proof of our main result is given in Section \ref{sec:proofs}, as well as our application to the contracting Lorenz map. 
Finally, in Section \ref{sec:conclusion} we provide an outlook and some concluding remarks, including a conjecture regarding NIO for the contracting Lorenz flow. 

\section{Setting and statement of results}\label{sec:setting}

\subsubsection*{Basic setup}

Throughout, we consider random perturbations of a fixed deterministic
skew product map $F : [-1, 1]^{d + 1} \circlearrowleft$ exhibiting nonuniform hyperbolicity, where $d \geq 1$. The mapping $F$ is of the form
\[
F(x, y) = (T(x), G(x,y))
\]
where $T : [-1,1]\circlearrowleft$ (the \emph{base dynamics}) and $G : [-1,1] \times [-1,1]^d \to [-1,1]^d$ (the \emph{fiber dynamics}) are mappings on their respective domains. Precisely: 
\begin{itemize}
	\item[(i)] $x \mapsto T(x)$ is piecewise\footnote{It is important to note that we will allow discontinuities in $x \mapsto T(x)$ and $x \mapsto G(x,y)$ to accommodate our intended application to the contracting Lorenz map. This lack of continuity generates some issues we deal with throughout our treatment of the abstract framework. } smooth with respect to a finite partition of $[-1,1]$ into disjoint intervals, while $\{ T' = 0 \}$ is finite and $T_* \Leb \ll \Leb$ (here, $T_*$ is the pushforward of a measure and $\Leb = \Leb_{[-1,1]}$ is Lebesgue measure on $[-1,1]$). 
	Additionally, $\log |T'(x)| \in L^1(dx)$. 
	
	\item[(ii)] $(x,y) \mapsto G(x,y)$ is piecewise smooth with respect to a finite partition of $[-1,1]^{d + 1}$ into finitely many disjoint measurable sets with nonempty interior, while for each fixed $x \in [-1,1]$ we have that 
	$G(x, \cdot) : [-1,1]^d \to [-1,1]^d$ is a local diffeomorphism onto its image. 
\end{itemize}

To define our perturbations, given $\omega = (\omega^1, \cdots, \omega^{d + 1}) \in \R^{d + 1}$, we set
\[
F_\omega(x, y) = F(x,y) + \omega \text{ mod 2} \, , 
\]
where the componentwise operation ``mod 2'' translates a point $r \in \R$ to 
its equivalence class $r$ mod 2 $\in (-1,1]$. Given 
a sequence $\uo = (\omega_1, \omega_2, \cdots)$ and $n \geq 1$, we consider the random compositions
\[
F^n_\uo = F_{\omega_n} \circ \cdots \circ F_{\omega_2} \circ F_{\omega_1} \, .
\]

To define the probabilistic law our perturbations take, let $\rho : \R\to \R_{\geq 0}$
be a BV density (what we refer to as a \emph{mother kernel}) with $\operatorname{Supp} \rho = [-1,1]$ (in particular, $\rho > 0$ almost-everywhere on $[-1,1]$), and for $\xi \in \R_{> 0}$ define 
\[
\rho^\xi(x) = \frac{1}{\xi} \rho\left(\frac{x}{\xi} \right) \, .
\]
We interpret the value $\xi$ as a \emph{noise amplitude}, and assume throughout
that $\omega_1, \omega_2, \cdots$ are IID $\R^{d + 1}$-valued random variables, 
each component of which is distributed like $\rho^\xi$. 
\subsubsection*{Sufficient conditions for noise-induced order}

For the fiber dynamics, we assume the following fiber-contraction property. 
\begin{itemize}
	\item[(F)] There is a constant $c > 0$ such that 
	$G(x, \cdot): [-1,1]^d \circlearrowleft$ satisfies 
	\[
	\Lip (G(x, \cdot)) \leq c < 1
	\]
	for each fixed $x \in [-1, 1]$. 
\end{itemize}

For the base dynamics, we will assume the following: 

\begin{itemize}
	\item[(B)(i)] (Deterministic dynamics) The deterministic base 
	dynamics $T : [-1,1] \circlearrowleft$ admits a unique, ergodic, absolutely continuous invariant measure $\mu_0$ with density $f_0$ which is $> 0$ almost everywhere. 
\end{itemize}
Under this assumption, the Lyapunov exponent
	\[
	\lambda_{\rm base}(0) := \lim_n \frac1n \log |(T^n)'(x)|
	\]
exists and is $x$-independent for Leb. almost-every $x$ by the Birkhoff ergodic theorem. 
We will additionally assume:
\begin{itemize}
	\item[(B)(ii)] (Positive LE) We have $\lambda_{\rm base}(0) > 0$. 
\end{itemize}

Our next assumptions refer to the base dynamics in the presence of noise. 
Below, given $\eta \in \R$ 
we write $T_\eta (x) = T(x) + \eta \text{ mod 2}$. For $\xi > 0$, we let $\eta_1, \eta_2, \cdots$ be an IID sequence distributed with law $\rho^\xi$. Given a sequence $\ue = (\eta_1, \eta_2, \cdots)$, we write
\[
T^n_\ue = T_{\eta_n} \circ \cdots \circ T_{\eta_1} \, .
\]

\begin{itemize}
	\item[(B)(iii)] For every $\xi > 0$ the Markov chain
	\[
	X_n := T_{\eta_n}(X_{n-1}) = T^n_\ue(X_0)
	\]
	on $[0,1]$ admits a unique stationary measure
	$\mu_\xi$ with density $f_\xi > 0$ Leb. almost everywhere. Moreover, this Markov chain is 
	\emph{exponentially mixing in} $L^1(dx)$: that is, for all $\xi > 0$ there exist $C_\xi, \gamma_\xi > 0$ such that if 
	$g_0 \in L^1(dx)$ is an arbitrary density on $[-1,1]$ and $g_n$ denotes the density of the law of 
	$X_n$ where $X_0$ has law $g_0$, then 
	\[
	\| g_n - f_\xi\|_{L^1} \leq C_\xi e^{- \gamma_\xi n} \,. 
	\]
	\end{itemize}
	
	Let $\lambda_{\rm base}(\xi)$ denote the Lyapunov exponent of $T^n_\ue$, which as we confirm below
exists and is almost-surely constant over typical initial conditions and with probability 1 for all $\xi > 0$. 

\smallskip
	\begin{itemize}
	\item[(B)(iv)] The Lyapunov exponent in the base is continuous at $0$
	with respect to the noise size $\xi$, i.e., $\lambda_{\rm base}(\xi) \to \lambda_{\rm base}(0)$ as $\xi \to 0$. 
\end{itemize}

Recall that a stationary measure $\mu$ for a Markov chain $(X_n)$ is \emph{ergodic}
if $(X_n)$-invariant sets have $\mu$-measure $0$ or $1$; here, a set $A \subset [0,1]$ is 
called $(X_n)$-invariant if with probability 1, $X_0 \in A$ if and only if $X_1 \in A$ 
\cite{Kifer}. 

The last assumption we make is to ensure that in the large-noise limit, 
the base dynamics $T$ experiences contraction on a large proportion of its domain. 
\begin{itemize}
	\item[(C)] We have
	\[
	\int_{-1}^1 \log |T'(x)| dx < 0 \,. 
	\]
\end{itemize}

As we show below (Lemma \ref{lem:zeroNoiseLEExists}), assumptions (F), (B) and (C) imply that for the deterministic map
$F$, the Lyapunov exponent
\[
\lambda(0) = \lim_n \frac1n \log \| D_{(x, y)}F^n \|
\]
exists and is constant over Lebesgue-typical $(x, y) \in [-1,1]^{d + 1}$, while 
at positive noise $\xi > 0$, the limit
\[
\lambda(\xi) = \lim_n \frac1n \log \| D_{(x,y)} F^n_\uo\|
\]
exists and is constant over Lebesgue-typical $(x, y) \in [-1,1]^{d + 1}$ and 
with probability 1 (Corollary \ref{cor:lyapRegSkewProd}). Roughly speaking, we will show the following: 
\begin{itemize}
\item[(a)] $\xi \mapsto \lambda(\xi)$ is continuous
\footnote{We actually only prove continuity of $\lambda$ where $\lambda_{\rm base}(\xi) > 0$, 
but this is sufficient for our purposes; see Proposition \ref{prop:LEinFiber} for details}. 
\item[(b)] $\lambda(0) > 0$ 
\item[(c)] $\limsup_{\xi \to \infty} \lambda(\xi) < 0$ 
\end{itemize}
Taken together, these imply noise-induced order, i.e., the
 existence of a transition from $\lambda(\xi) > 0$ to 
$\lambda(\xi) < 0$. Precisely, we have the following
\begin{Theorem}[Sufficient condition for noise-induced order] \label{thm:NIO}
Under assumptions (F),(B) and (C) above, there exist noise amplitudes $\xi_+ < \xi_-$ 
such that $\lambda(\xi) > 0$ for $\xi \in [0,\xi_+)$ and $\lambda(\xi) < 0$ for $\xi \in (\xi_-, \infty)$. 
\end{Theorem}

\begin{Remark}
As showed by numerical experiments for unimodal maps in \cite{Ni} and 
through a rigorous computed aided proof in Lasota-Mackey 
maps \cite{ChiGaNiSa} there may be more than one transition from positive to negative.
Our result proves that there exists at least one such a transition.
\end{Remark}


%
%

\subsubsection*{Application to contracting Lorenz map}
We will apply our results to skew-products of the form 
\[
F(x, y) = (T(x), G(x, y))
\]
where
\[
T(x)=sgn(x)(\rho |x|^s-1), \quad G(x,y)= 2^{-r } sgn(x)y |x|^r+c
\]
which arise naturally as the first return maps for the contracting Lorenz Flow.
We will show (Section \ref{subsec:contractLorenzMap}) that exist values of $\rho$ and $s$ for which the top Lyapunov
exponent transitions from positive to negative as the noise size increases.
As $s$ increases the size of the contracting part of the phase space grows; a plot of the map
$T$ for parameters that present NIO can be found in figure \ref{fig:zeroline}. 

\begin{figure}[h]
	\begin{center}\label{fig:T}
	\includegraphics[width=7cm]{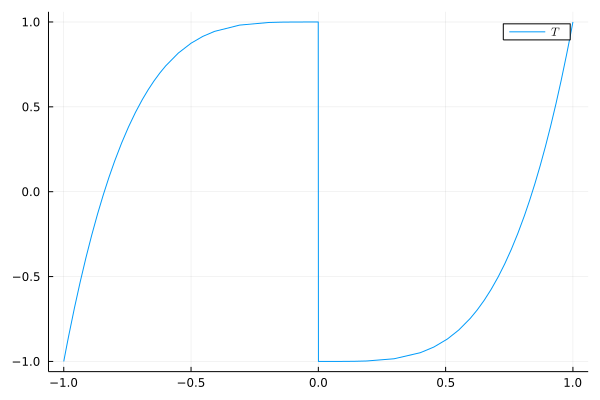}
	\end{center}
	\caption{The map $T$ for $s=4$, $\rho =2$.}
\end{figure}


\section{Proofs}\label{sec:proofs}

\subsection{Existence of Lyapunov exponent}

\subsubsection{Deterministic case ($\xi = 0$)}
We begin by addressing existence of the Lyapunov exponent $\lambda(0)$ 
for $F$ in the absence of noise. 

\begin{Lemma}\label{lem:zeroNoiseLEExists}
The limit
\[
\lambda(0) = \lim_n \frac1n \log \| D_{(x,y)} F^n\|
\]
exists and is constant ($(x,y)$-independent) for Leb-almost every $(x,y) \in [-1,1]^{d+1}$, and moreover, coincides with $\lambda_{\rm base}(0)$. 
\end{Lemma}
\begin{proof}
The limit defining $\lambda_{\rm base}(0)$ exists and is given by
\[
\lambda_{\rm base}(0) = \int_{-1}^1 \log |T'(x)| d \mu_0(x)
\]
by the Birkhoff ergodic theorem and ergodicity of $\mu_0$. 
To lift this to convergence of the full exponent $\lambda(0)$, we start by computing
$D F^n$, setting notation we'll use throughout.
Fixing $x$, let $D_{(x,y)} G$ denote the Jacobian of the mapping $G(x, \cdot) : [-1,1]^d \circlearrowleft$, and fixing $y$, let $\nabla G(x,y)$ denote the vector in $\R^d$ of partial
derivatives of the components of $G$ with respect to $x$. 

In this notation, for the full Jacobian of $F : [-1,1]^{d+1} \circlearrowleft$ we have
\[
D_{(x,y)} F = \begin{pmatrix} T'(x) & 0 \\ \nabla G(x,y) & D_{(x,y)} G \end{pmatrix} \,. 
\]
Writing $D_{(x,y)} G^n = D_{F^{n-1}(x,y)} G \circ \cdots \circ D_{(x,y)} G$, we see that
\[
D_{(x,y)} F^n = \begin{pmatrix} (T^n)'(x) & 0 \\
(*) & D_{(x,y)} G^n \end{pmatrix} \, , 
\]
where
\begin{align*}
(*) & = (T^{n-1})'(x) \nabla G \circ F^{n-1} (x,y) + (T^{n-2})'(x) D_{F^{n-1}(x,y)} G \nabla G(F^{n-2}(x,y)) + \cdots \\
& + (T^{n-i})'(x) D_{F^{n-(i-1)}(x,y)} G^{i-1} \nabla G(F^{n-i}(x,y)) + \cdots \\
& + D_{F(x,y)} G^{n-1} \nabla G(x,y) \,. 
\end{align*}

We now set about estimating $\lambda(0)$. To start, let $x$ be drawn from the Leb. typical set for which $\lim_n \frac1n \log |(T^n)'(x)| = \lambda_{\rm base}(0)$, and let $y \in [-1,1]^d$ be arbitrary. For a vector $v = (a, w) \in \R^{d+1}, a \in \R, w \in \R^d$, we use the lower bound
\[
\| D_{(x,y)} F^n(v) \| \geq |a| |(T^n)'(x)| \,, 
\]
which implies immediately that $\liminf_n \frac1n \log \| D_{(x,y)} F^n\| \geq \lambda_{\rm base}(0)$ for all such $(x,y)$. For the upper bound, fix $\epsilon > 0$ and let $C = C(\epsilon, x) \geq 1$ be sufficiently large so that $|(T^n)'| \leq C e^{n (\lambda_{\rm base}(0) + \epsilon)} $
for all $n$. Let $M = \sup_{(x,y)} \| \nabla G\|$. Then, 
\begin{align*}
\| D_{(x,y)} F^n(1, 0)\| & \leq C M ( e^{(n-1) (\lambda_{\rm base}(0) + \epsilon)} 
+ c e^{(n-2)(\lambda_{\rm base}(0) + \epsilon)}  \\
& + \cdots + c^{n-2} e^{\lambda_{\rm base}(0) + \epsilon} + c^{n-1}) \\
&= C M e^{(n-1) (\lambda_{\rm base}(0) + \epsilon)} (1 + c + \cdots + c^{n-1}) \\
& = C M e^{(n-1) (\lambda_{\rm base}(0) + \epsilon)} \cdot \frac{1-c^n}{1 - c}
\end{align*}
while for $w \in \R^d$, 
\[
\| D_{(x,y)}F^n (0, w) \| \leq c^n \| w \| \, .
\]
We conclude
\[
\limsup_n \frac1n \log \| D_{(x,y)} F^n\| \leq \lambda_{\rm base}(0) + \epsilon \, , 
\]
and the proof is complete on taking $\epsilon \to 0$.

\end{proof}

\subsubsection{Noisy case ($\xi > 0$)}
Before examining Lyapunov exponents, we briefly recall two alternative formulations
of the random dynamics $(F^n_\uo)$. The first is as a Markov chain $(X_n, Y_n)$ on $[-1,1]^{d + 1}$ defined by
\[
(X_n, Y_n) = F_{\omega_n}(X_{n-1}, Y_{n-1}) \, .
\]
Recall that the process $(X_n)$ is a Markov chain in its own right (notation as in assumption (B)(iii)). 

The second alternative formulation is as a ``deterministic skew product''. 
For this, let $\Omega = (\R^{d + 1})^{\otimes \mathbb N}$ be the sequence space
of random samples $\uo  = (\omega_n)_{n \geq 1}$. Let $\Fc$ be the corresponding Borel $\sigma$-algebra and for $\xi > 0$, let $\P_\xi$ denote the probability measure on $(\Omega, \Fc)$ assigning law $\rho_\xi$ to each real coordinate. We define the (ergodic) mpt
$\theta : (\Omega, \Fc, \P_\xi)$ to be the leftward shift, given for $\uo = (\omega_1, \omega_2, \cdots)$ by 
\[
\theta \uo = (\omega_2, \omega_3, \cdots) \,. 
\]
Write $\tau : \Omega \times [-1,1]^{d + 1} \circlearrowleft$ for the 
skew product system
\[
\tau(\uo; x,y) = (\theta \uo; F_{\omega_1}(x,y)) \, , 
\]
so that $\tau^n(\uo; x,y) = (\theta^n \uo; F^n_\uo(x,y))$. 
Similarly, we write $\tau_1$ for the corresponding 
skew product on $\Omega \times [-1,1]$ tracking only the $x$-coordinate; that is, 
\[
\tau_1(\uo; x) = (\theta \uo, T_{\omega_1^1}(x)) \, . 
\]

Turning attention to Lyapunov exponents, we start with the dynamics in the base: 
\begin{Lemma}\label{lem:baseLECty} \ 
\begin{itemize}
\item[(a)] Assume condition (B)(iii). Then, for all $\xi \in (0,\infty)$, the limit
\[
\lambda_{\rm base}(\xi) = \lim_n \frac1n \log |(T^n_\ue)'(x)|
\]
exists and is constant (independent of $x$) for Lebesgue-almost every $x \in [-1,1]$ and
with probability 1. 
\item[(b)] If all of condition (B) holds, then we have that $\xi \mapsto \lambda_{\rm base}(\xi)$ is continuous over $\xi \in [0,\infty)$. 
\end{itemize}
\end{Lemma}
\begin{proof}
For (a), the proof is to apply the Birkhoff ergodic theorem to the measure-preserving transformation
 $\tau_1 : \Omega \times [-1,1] \circlearrowleft$ with invariant measure $\P_\xi \times \mu_\xi$, using the well-known fact that 
$\mu_\xi$ is an ergodic stationary measure iff $\P_\xi \times \mu_\xi$ is an ergodic invariant
measure for $\tau_1$ \cite{Kifer}.

For (b), (B)(iv) ensures continuity of $\xi \mapsto \lambda_{\rm base}(\xi)$ at $\xi = 0$. For $\xi > 0$, 
the Birkhoff ergodic theorem implies
\begin{align}\label{eq:statFormBaseLE}
\lambda_{\rm base}(\xi) = \int \log |T'(x)| f_\xi(x) dx 
\end{align}
for all $\xi \in [0,\infty)$. In view of the fact that $\log |T'(x)| \in L1(dx)$, it suffices to check
that $\xi \mapsto f_\xi$ varies continuously in the BV norm. The following argument is 
standard, repeated below for the sake of completeness (see, e.g., \cite{Ni}). 

Recall that if $X_0$ is distributed like some $L^1$ density $f$, then 
the law of $X_1$ is given by 
\begin{align}\label{eq:formulaForSD}
\mathcal L_\xi f :=  \rho^\xi \hat{*} \mathcal L_T f \, , 
\end{align}
where $ \mathcal L_T$ is the transfer operator of $T$, given by
\[
\mathcal L_T f (x) = \sum_{y \in T^{-1} x} \frac{f(y)}{|T'(y)|}
\]
for $f : [-1,1] \to \R_{\geq 0}$, and $\rho^\xi \hat{*}$ denotes the periodic convolution defined by
\[
\rho^\xi \hat{*} f (x) = \sum_{i \in \mathbb Z}  \int_{-1}^1 \rho(x+2i-y) f(y)dy.
\]

In particular, $\rho$ is a BV density supported in $[-1,1]$, as in the proof of Lemma \ref{lem:infiniteNoiseBase}
we can check that ${\mathcal L}_\xi$ is bounded in norm as an operator $L^1 \to BV$. We see, then, that
for mean-zero $g \in L^1(dx)$, 
\[
\| \mathcal L^{n + 1}_\xi g \|_{BV} \leq C_\xi e^{- n \gamma_\xi} \| \mathcal L_\xi\|_{L^1 \to BV} \| g \|_{L^1} \, .
\]
Let now $\xi, \xi' > 0$ and fix $n$ so that $C_\xi e^{- n \gamma_\xi} \| \mathcal L_\xi\|_{L^1 \to BV} < 1/4$. We estimate
\begin{align*}
\| f_\xi - f_{\xi'} \|_{BV} & = \| \mathcal L^{n + 1}_\xi f_\xi - \mathcal L^{n + 1}_{\xi'} f_{\xi'} \|_{BV} \\
& \leq \| \mathcal L^{n + 1}_\xi (f_\xi - f_{\xi'}) \|_{BV} + \| (\mathcal L^{n + 1}_\xi - \mathcal L^{n + 1}_{\xi'}) f_{\xi'} \|_{BV}
\end{align*}
The first term is $\leq (1/4) \| f_\xi - f_{\xi'}\|_{L^1} \leq (1/2) \| f_\xi - f_{\xi'}\|_{L^1}$, and so overall
\[
\| f_\xi - f_{\xi'}\|_{BV} \leq 2 \| \rho_\xi \hat{\ast} - \rho_{\xi'} \hat{\ast} \|_{L^1 \to BV} 
\]
using that $\mathcal L_T : L^1 \to L^1$ has norm 1 and that $\| f_{\xi'}\|_{L^1} = 1$. 
It is straightforward to check that the RHS goes to zero as $\xi' \to \xi$, completing the proof. 
\end{proof}

For the random dynamics in the full skew product, we start by checking 
existence and uniqueness of stationary measures $\nu_\xi$ for the Markov chain $(X_n, Y_n)$. 

\begin{Lemma}\label{lem:uniqueStatMeasForSkew}
Let $\xi > 0$ be arbitrary. Assume the fiber contraction condition (F) and that $(X_n)$ has a unique (absolutely continuous and ergodic) stationary measure $\mu_\xi$ (as in assumption (B)(iii)).  Then, the Markov chain $(X_n, Y_n)$ admits a unique, ergodic, absolutely continuous stationary measure $\nu_\xi$.
\end{Lemma}
\begin{proof} 
Any stationary measure for $(X_n, Y_n)$ is automatically absolutely continuous with a $BV$ density. Existence follows from the following mild variation of the typical Krylov-Bogoliubov argument. 
Given a density $h$ on $[-1,1]^{d + 1}$, let $P^* h$ denote the law of $(X_1, Y_1)$ assuming
$(X_0, Y_0)$ is distributed like $h dx dy$. Fixing a smooth initial density $h$, consider the sequence
\[
h_n := \frac1n \sum_{i =0}^{n-1} (P^*)^i h \, , 
\]
noting $(P^*)^i h$ is the law of $(X_i, Y_i)$ assuming $(X_0, Y_0)$ is distributed like $h dx dy$. 
By compactness of $BV$ in $L^1$, there is an $L^1$-convergent
 subsequence $h_{n_k}$ with limit $h \in L^1$. That $h$ is an invariant density now follows 
 from the straightforward bound $\| P^* h - P^* h_{n_k} \|_{L^1} \leq \| h - h_{n_k}\|_{L^1}$. 

In pursuit of a contradiction to uniqueness, 
fix two distinct stationary measures $\nu, \nu'$ for $(X_n, Y_n)$ and 
a continuous observable $\varphi : [-1,1]^{d + 1} \to \R$ for which $\int \varphi d \nu \neq \int \varphi d \nu'$. Note that $\nu, \nu'$ project to the unique stationary measure $\mu_\xi$ for $(X_n)$ on the $x$-coordinate. Moreover, by the Birkhoff ergodic theorem applied to $\tau : \Omega \times [-1,1]^{d + 1} \circlearrowleft$, the limit
\[
\varphi_*(\uo; x,y) = \lim_n \frac1n \sum_0^{n-1} \varphi \circ F^n_\uo(x,y)
\]
exists for $\nu$-a.e. $(x,y)$ and for $\nu'$-a.e. $(x,y)$ (recall that $\nu$, resp. $\nu'$, is stationary iff $\P_\xi \times \nu$, resp. $\P_\xi \times \nu'$, is invariant for $\tau$). 
The limit function $\varphi_*(\uo; x,y)$ satisfies
\[
\int \varphi_*(\uo; x,y) d \P_\xi(\uo) d \nu = \int \varphi d \nu
\]
and the analogous statement for $\nu'$.

We claim that if the limit defining $\varphi_*(\uo; x,y)$ exists for some $\uo, x$ and $y$, 
then it exists and coincides with $\varphi_*(\uo; x, y')$ for all $y' \in [-1,1]^d$. 
To conclude from here, we observe that
\[
\int \varphi d \nu = \int \varphi_*(\uo; x, 0) d \P_\xi(\uo) d \nu(x,y) = \int \varphi_*(\uo; x, 0) d \P_\xi(\uo) d \mu_\xi(x) 
\]
using that $\nu$ projects to $\mu_\xi$ on the $x$-coordinate. We derive an identical expression for $\int \varphi d \nu'$, which leads to a contradiction. 

To check the claim, assume $\varphi_*(\uo; x,y)$ exists and let $y' \in [-1,1]^d$ be arbitrary. 
Observe that for all $n \geq 0$, 
\[
|F^n(x,y) - F^n(x,y')| \leq c^n |y - y'| 
\]
(recall that $c < 1$ is an upper bound on $\Lip(G(x, \cdot))$ viewed as a function on $[-1,1]^d$; see assumption (F)). 
With $\epsilon > 0$ fixed, let $\delta > 0$ be such that $|\varphi(x,y) - \varphi(x,y')| < \epsilon$ if $|y - y'| < \delta$. Fix $N$ large enough so that $c^N |y - y'| < \delta$. Then, 
\[
\left| \frac1n \sum_0^{n-1} \varphi \circ F^n_\uo(x,y) - \frac1n \sum_0^{n-1} \varphi \circ F^n_\uo(x,y') \right| \leq \frac{n - N}{n} \epsilon + \frac{2 N}{n} \| \varphi \|_\infty 
\]
for all $n \gg N$. Taking $n \to \infty$ and $\epsilon \to 0$ establishes the claim. 

At this point, we shown unique existence of a stationary measure $\nu_\xi$ for the Markov chain
$(X_n, Y_n)$, and it remains to check ergodicity. To see this, assume $A$ is an invariant set with 
$\nu_\xi(A) \in (0,1)$, and form the probability measures
\[
\nu_1(K) = \frac{\nu_\xi(A \cap K)}{\nu_\xi(A)} \, , \quad \nu_2(K) = \frac{\nu_\xi(K \cap A^c)}{\nu_\xi(A^c)} 
\]
for $K \subset [-1,1]^{d + 1}$ measurable. It is straightforward to check that invariance of $A$ implies $\nu_1, \nu_2$ are distinct stationary measures
for $(X_n, Y_n)$, contradicting uniqueness. We conclude $\nu_\xi$ is ergodic. 
\end{proof}

\begin{Corollary}\label{cor:lyapRegSkewProd}
Assume the setting of Lemma \ref{lem:uniqueStatMeasForSkew}. For all $\xi > 0$, the limits
\[
\chi_i(\xi) = \lim_n \frac1n \log \sigma_i(D_{(x,y)} F^n_\uo) \, , \quad 1 \leq i \leq d + 1
\]
exist and are constant over $\nu_\xi$-typical $(x, y) \in [-1,1]^{d + 1} $ with probability 1 (some possibly equal to $-\infty$). Moreover, the limit defining
\[
\lambda(\xi) := \chi_1(\xi) 
\]
exists and is constant over $\Leb$-typical $(x,y) \in [-1,1]^{d + 1}$ with probability 1. 
\end{Corollary}
Here, 
$\sigma_i$ refers to the $i$th singular value of a matrix. The values $\{ \chi_i\}$
 are the \emph{Lyapunov exponents} of the derivative cocycle $D_{(x,y)}F^n_\uo$.
We set 
\[
\lambda(\xi) = \chi_1(\xi)\, , 
\]
 the top Lyapunov exponent. 

\begin{proof}
For $\xi > 0$ and at $\nu_\xi$-typical $(x,y)$, everything follows from the subadditive ergodic theorem applied to the sequence of functions 
\[
(\uo, (x, y)) \mapsto \| \wedge^k D_{(x, y)} F^n_\uo\| = \prod_{i = 1}^k \sigma_i(D_{(x,y)} F^n_\uo) \,, \quad n \geq 1
\]
for each fixed $k$, viewed as subadditive over the dynamical system $\tau : \Omega \times [-1,1]^{d + 1} \circlearrowleft$ with invariant measure $\P_\xi \times \nu_\xi$. 

It remains to check that at $i = 1$, this convergence holds for Lebesgue-typical $(x,y)$. So as not to interrupt the flow of ideas, we carry this argument out in the Appendix (Section \ref{sec:appConvLE}).
\end{proof}

Also of interest for us are the Lyapunov exponents in the invariant 
bundle $\{ 0 \} \times \R^d$ tangent to the fibers. 
Equivalently, these are the Lyapunov exponents of the 
 the cocycle 
 \[
B^n_{\uo; (x,y)}:= D_{F^{n-1}_\uo(x,y)} G \circ \cdots \circ D_{F_{\uo} (x,y)} G \circ D_{(x,y)} G
\]
on $\R^d$ (viewed as a cocycle over $\tau : \Omega \times [-1,1]^{d+1} \circlearrowleft$). 
The following is immediate from the subadditive ergodic theorem. 

\begin{Corollary}
Assume the setting of Lemma \ref{lem:uniqueStatMeasForSkew}. For all $\xi > 0$, the limits
\[
\hat \chi_i(\xi) = \lim_n \frac1n \log \sigma_i(D_{(x, y)} G^n_\uo) \, , \quad 1 \leq i \leq d \, .
\]
exist and are constant over $\nu_\xi$-typical $(x, y) \in [-1,1]$ with probability 1 (some possibly equal to $-\infty$). 
\end{Corollary}
Note that the fiber contraction assumption (F) implies $\hat \chi_1 \leq \log c$ holds
for all $\xi$. 

\subsection{Proof of Theorem \ref{thm:NIO}}

The main step is to affirm the following formula for the top Lyapunov exponent $\lambda(\xi)$ of $F^n_\uo$. 
\begin{Proposition}\label{prop:LEinFiber}
For all $\xi \geq 0$ we have that
\[
\lambda(\xi) = \max\{ \lambda_{\rm base}(\xi), \hat \chi_1(\xi)\} \, . 
\]
\end{Proposition}
\begin{proof}
Below, we suppress $\xi$-dependence, writing $\chi_i = \chi_i(\xi), \lambda_{\rm base} = \lambda_{\rm base}(\xi)$, etc. 

Let $V = \{ 0 \} \times \R^d$ denote the linear span of the last $d$ coordinates in $\R^{d + 1}$. By the skew product structure of $D F$, we have for $v \in V$ that
\[
D_{(x,y)} F_\uo^n (0, v) = B^n_{\uo; (x,y)} (v) \, .
\]
In particular, 
\[
\det(D_{(x,y)} F^n_\uo|_V) = \det(B^n_{\uo; (x,y)}) \,. 
\]
By Corollary \ref{cor:subspaceGrowth} applied to $D F^n_\uo$ and $B^n_{\uo; (x,y)}$, we see that there exist indices $1 \leq i_1 < \cdots < i_{d} \leq d + 1$
such that
\[
\chi_{i_1} + \cdots + \chi_{i_d} = \hat \chi_1 + \cdots + \hat \chi_d \,. 
\]
To obtain another relation, observe by the block diagonal structure of $D F^n_\uo$
that 
\[
\det(D_{(x,y)} F^n_\uo) = (T^n_\ue)'(x) \det(B^n_{\uo; (x,y)})
\]
where $\ue = (\eta_1, \eta_2, \cdots)$ and $\eta_i := \omega_i^1$. We obtain immediately
that
\[
\chi_1 + \cdots + \chi_{d + 1} = \lambda_{\rm base} + \hat \chi_1 + \cdots + \hat \chi_d \,, 
\]
and conclude
\[
\lambda_{\rm base} = \chi_{i_*}
\]
where $i_*$ is the unique element of $\{ 1 , \cdots, d+1\} \setminus \{ i_1, \cdots, i_d\}$. 

For the remaining exponents $\chi_{i_j}, 1 \leq j \leq d$, let $e_1, \cdots, e_d$ be any 
basis for $V$ and let $V_j = \operatorname{Span}\{ e_1, \cdots, e_j\}$, noting $V_1 \subsetneq V_2 \subsetneq \cdots \subsetneq V_d = V$. Iteratively applying Corollary \ref{cor:subspaceGrowth} we see that there is a permutation $\sigma$ of $\{ 1, \cdots, d\}$ such that
\[
\lim_n \frac1n \log | \det(D_{(x,y)} F^n_\uo|_{V_j})| = \chi_{i_{\sigma(1)}} + \cdots + \chi_{i_{\sigma(j)}}
\]
for each $1 \leq j \leq d$. 
Similarly, there is another permutation $\hat \sigma$ of $\{ 1, \cdots, d\}$ such that
\[
\lim_n \frac1n \log |  \det(B^n_{\uo; (x,y)}|_{V_j}) | = \hat \chi_{\hat \sigma(1)} + \cdots + \hat \chi_{\hat \sigma(j)} \,. 
\]
Since $\det(D_{(x,y)} F^n_\uo|_{V_j}) = \det(B^n_{\uo; (x,y)}|_{V_j})$ for all $j$, we conclude
\[
\chi_{i_{\sigma(j)}} = \hat \chi_{\hat \sigma(j)} \quad \text{ for all } 1 \leq j \leq d \,. 
\]
In summary, we have shown
\begin{itemize}
	\item $\lambda_{\rm base} = \chi_{i_*}$ for some $i_* \in \{ 1, \cdots, d + 1\}$. 
	\item The list of remaining exponents $\{ \chi_{i} : 1 \leq i \leq d + 1, i \neq i_*\}$
	coincides with the list $\{ \hat \chi_i : 1 \leq i \leq d\}$, counting multiplicities. 
\end{itemize}
We conclude that $\max\{ \chi_i \} = \max\{\lambda_{\rm base}, \hat \chi_1\}$ as desired. 
\end{proof}

We will also require the following on the behavior of $\lambda_{\rm base}(\xi)$ in the infinite-noise limit $\xi \to \infty$. 
\begin{Lemma}\label{lem:infiniteNoiseBase}
Under assumption (B)(iii), we have $\lim_{\xi \to \infty} \lambda_{\rm base}(\xi) = \frac{1}{2}\int_{-1}^1 \log |T'(x)| d x$. 
If in addition (C) holds, then $\lim_{\xi \to \infty} \lambda_{\rm base}(\xi) < 0$ 
\end{Lemma}
\begin{proof}
Recall (equation \eqref{eq:formulaForSD}) that the density $f_\xi$ of the stationary measure $\mu_\xi$ for the Markov chain $(X_n)$ on $[-1,1]$
satisfies
\begin{align*}
f_\xi = \mathcal L_\xi f_\xi = \rho^\xi \hat{*} \mathcal L_T f_\xi
\end{align*}
\newcommand{\Var}{\operatorname{Var}}

As we will argue below, it suffices to prove that for all $f \in L^1([-1,1])$, we have
\begin{align}\label{eq:conv12}
\Var(\rho^\xi \hat{*} f) \to 0 \quad \text{ in }BV \,. 
\end{align}
where $\Var$ denotes the first variation of a function on $[-1,1]$. 
Assume that $\rho$ is $C^1$, and 
recall that for $\phi \in C^1([-1,1])$, 
$\Var(\phi) = \int_{-1}^1 |\phi'(x)| dx$; the result when $\rho$ is not $C^1$ follows from 
a straightforward approximation argument. 

We now estimate 
\begin{align*}
\Var(\rho^\xi \hat{*} f) & = \int_{-1}^1 |(\rho^\xi \hat{*} f)'| dx = \int_{-1}^1 \left| \sum_{i \in \mathbb Z} \int_{-1}^1 (\rho^\xi)'(x + 2 i - y) f(y) dy\right| dx \\
&\leq  \sum_{i\in \mathbb{Z}}\int_{-1}^1 \left(\int_{-1}^1 |(\rho^\xi)'(x+2i-y)| dx\right) |f(y)|dy \\
&\leq \sum_{i\in \mathbb{Z}}\int_{-1}^1 \Var_{[-1+2i-y, 1+2i-y]}(\rho^\xi)|f(y)| dy \\
&\leq  \sum_{i\in \mathbb{Z}}\int_{-1}^1 \Var_{[-2+2i, 2+2i]}(\rho^\xi)|f(y)| dy  
\end{align*}
having used in the last line that $\Var_I (\phi) \leq \Var_J(\phi)$ if $I, J$ are compact intervals such that $I \subset J$ and $\phi : J \to \R$ is BV. In total, 
\[
\Var(\rho^\xi\hat{\ast}f) \leq \sum_i \Var_{[-2+2i, 2+2i]}(\rho^\xi) \|f\|_{L^1} = 2\Var_{[-\xi,\xi]}(\rho^\xi)\|f\|_{L^1}
\]
recalling that $\rho$ is supported on $[-1,1]$, hence $\rho^\xi$ is supported on $[-\xi, \xi]$. 

From here, we note that
\[
\Var_{[-\xi,\xi]} (\rho^\xi) = \int_{-\xi}^\xi (\rho^\xi)'(x) dx = \frac{1}{\xi^2} \int_{-\xi}^\xi \rho'(x / \xi) dx = \frac{1}{\xi} \Var_{[-1,1]}(\rho) \, , 
\]
and conclude 
\[
\Var(\rho^\xi \hat{*} f) \leq \frac{2 \Var_{[-1,1]}(\rho)}{\xi} \|f\|_{L^1} \to 0  \quad \text{ as } \xi \to \infty \, . 
\]

As $\| f_\xi\|_{L^1} \equiv 1$ for all $\xi$, equation \eqref{eq:formulaForSD} now implies
 $f_\xi$ converges to $1/2$ in $BV$. In view of the fact that $\log |T'| \in L^1(m)$ and equation \eqref{eq:statFormBaseLE}, it follows from $\lambda_{\rm base}(\xi) \to \frac{1}{2}\int_{-1}^1 \log|T'(x)| dx$
 as $\xi \to \infty$, which is $< 0$ by assumption (C). 
 \end{proof}
 




We are now in position to close the proof. 

\begin{proof}[Proof of Theorem \ref{thm:NIO}]
The three ingredients we use are as follows: 
\begin{itemize}
    	\item[(a)] $\xi \mapsto \lambda_{\rm base}(\xi)$ is continuous over $\xi \in [0,\infty)$ (Lemma \ref{lem:baseLECty}(b)); 
	\item[(b)] $\lambda_{\rm base}(0) > 0$ (Assumption (B)(ii)) and $\lim_{\xi \to \infty} 
	\lambda_{\rm base}(\xi) < 0$ (Lemma \ref{lem:infiniteNoiseBase}); 
	\item[(c)] $\lambda(\xi) = \max\{ \lambda_{\rm base}(\xi), \hat \chi_1(\xi)\}$ (Proposition \ref{prop:LEinFiber}); and 
	\item[(d)] $\hat \chi_1(\xi) \leq \log c < 0$ for all $\xi$ (Assumption (F)).
\end{itemize}
It is immediate from (a) and (b) that $\exists 0 < \xi_+  < \xi_-$ such that $\lambda_{\rm base}(\xi) > 0$ for $\xi \in [0,\xi_+)$ and $\lambda_{\rm base}(\xi) < 0$ for all $\xi \in (\xi_-, \infty)$. 
From (c) and (d), it follows that $\lambda(\xi) > 0$ for $\xi \in [0,\xi_+)$ 
and $\lambda(\xi) < 0$ for $\xi \in (\xi_-, \infty)$. This completes the proof. 
\end{proof}

\subsection{Application to contracting Lorenz map}\label{subsec:contractLorenzMap}

For a complete construction of the contracting Lorenz Flow we refer to \cite{ArPa, GaNiPa}.
For the sake of this paper, what  matters is that the first return map to the Poincaré section 
$F:[-1,1]^2\to [-1,1]^2$ has the form
\[
F(x, y) = (T(x), G(x, y))
\]
where
\[
T(x)=sgn(x)(\rho |x|^s-1), \quad G(x,y)=2^{-r} sgn(x)y |x|^r+c
\]
for some $c$ such that the sets $F([-1,0]\times [-1, 1])$ and $F([0,1]\times [-1, 1])$
do not overlap and $0<\rho\leq 2$, with $r>s+3$.

The map $T$ satisfies the following properties:
\begin{itemize}
\item the order of $T'$ at $0$ is $s-1>0$,
\item $T$ has a discontinuity in $0$, $T(0^+)=-1$, $T(0^-)=1$,
\item $T'(x)>0$ for $x\neq 0$,
\item $\max |T'(x)|$ is attained at $-1$ and $1$,
\item $T$ has negative Schwarzian derivative.
\end{itemize}

The parameter $\rho$ can be chosen in such a way that the points
$1$ and $-1$ are preperiodic repelling (which is a Misiurewicz-type condition \cite{Mis}); 
this condition is always satisfied when $\rho = 2$ for any value of $s$. That both of $1,-1$ are preperiodic
repelling will be assumed in throughout in the following discussion. 

Under these hypotheses and parameter choices it was proved in \cite{Ro} that the flow 
associated to $F$ admits an attractor $\Lambda_0$.
Indeed, a stronger result is proved, i.e., that $2$ is the density point of a positive 
Lebesgue measure set of parameters (called the Rovella parameters) 
for which an attractor exists, c.f. Remark \ref{rmk:RovellaParams} below. 
In \cite{MeEx} it is proved that under the above conditions, the map 
$T$ has a unique a.c.i.m. and positive Lyapunov 
exponent, confirming conditions (B)(i) and (B)(ii). 
Condition (F) is evident, and so below we carry out the remaining work
of checking (B)(iii), (B)(iv) and (C). 

\subsubsection{Condition (B)(iii)}
It was checked directly in \cite{Me}
that $\exists \xi_0 > 0$ such that Condition (B)(iii) holds for all 
$\xi \in (0,\xi_0]$. It is straightforward to check that this 
can be promoted to all $\xi > 0$ by applying, e.g., the arguments
of Section 4 in \cite{Ni}, where general conditions were given for deducing $L^1$ contraction
of  $\mathcal L_{\xi'}$ when $\mathcal L_\xi$ is an $L^1$ contraction for some $0 < \xi < \xi'$. 

\subsubsection{Condition (B)(iv)}
In view of equation \eqref{eq:statFormBaseLE} 
relating $\lambda_{\rm base}(\xi)$ to the expectation of $\log |T'|$ with respect to $f_\xi d x$, 
it suffices to have $f_\xi \to f_0$ in $BV$. 
In \cite{Me}, it was shown that $f_\xi \to f_0$ in $L^1$, better known as \emph{strong stochastic
stability} of the map $T$-- however, this is not quite enough for our purposes. 
Fortunately, the tower construction given in \cite{Me, MePrior} provides enough information, as we
show below. 

The tower construction consists of an extension $\hat T : \hat I \circlearrowleft$
of the dynamic $T$, where $\hat I \subset \mathbb Z \times [-1,1]$
 is the union of a countable collection of sets of the form $E_k := \{ k \} \times B_k, k \in \mathbb Z$, 
 and $\{ B_k\}$ is a proscribed partition mod 0 of $[-1,1]$. 
Once $\hat T : \hat I \circlearrowleft$ is specified, it satisfies the entwining relation $ \pi \circ \hat T = T \circ \pi$, 
where $\pi : \hat I \to [-1,1]$ is the projection taking $(k, x) \in \{ k \} \times B_k$ to its corresponding point $x \in [-1,1]$. 

In \cite{Me}, a random dynamical system comprised of compositions of the form 
$\hat T_\eta : \hat I \circlearrowleft, \eta \in \R$ was constructed analogously to the 
construction of the random perturbations $T_\eta$ of the base map $T$. For each $\eta$, these
random systems entwine with $T_\eta$, i.e., $\pi \circ \hat T_\eta = T_\eta \circ \pi$. 
Below, we write $\hat m$ to denote the natural Lebesgue measure on $\hat I$. 
\begin{Proposition}[\cite{Me}]\label{prop:Metzger}
There exists $\xi_0 > 0$ such that for all $\xi \in [0, \xi_0)$, the random system $\hat T_\eta$ on $\hat I$
admits a unique BV stationary density $\hat f_\xi$ with respect to Lebesgue measure $\hat m$ on $\hat I$. 
Moreover, these densities have the property that 
\[
\hat f_\xi \to \hat f_0 \quad \text{ in } BV
\]
as $\xi \to 0$. 
\end{Proposition}

\begin{Corollary}
Condition (B)(iv) holds, i.e., $\lambda_{\rm base}(\xi) \to \lambda_{\rm base}(0)$ as $\xi \to 0$. 
\end{Corollary}
\begin{proof}
By the entwining property $\pi \circ \hat T_\eta = T_\eta \circ \pi$, it follows that $\pi_* ( \hat f_\xi d\hat m) = f_\xi dx$, 
and so 
\begin{align*}
\lambda_{\rm base}(\xi) & = \int_{-1}^1 \log |T'(x)| f_\xi dx = \int_{-1}^1 \log |T'(x)| \pi_* ( \hat f_\xi d \hat m)(x) \\
& = \sum_{k \in \mathbb Z} \int_{B_k} \log |T'(\pi(\hat x))| \hat f_\xi(\hat x) d \hat m(\hat x) \,. 
\end{align*}
for all $\xi \in [0,\xi_0]$, $\xi_0$ as in Proposition \ref{prop:Metzger}. 
The convergence $\lambda_{\rm base}(\xi) \to \lambda_{\rm base}(0)$ as in Condition (B)(iv) now follows
from BV convergence of $\hat f_\xi \to \hat f_0$ on exchanging the summation and limit. 
\end{proof}

\subsubsection{Condition (C)}
\begin{Lemma}
If $\rho=2$ and $s>2.67835$ then
\[
	\frac{1}{2}\int_{-1}^1\log(|T'|)dx<0
\]
\end{Lemma}
\begin{proof}
This follows from direct computation: 
\[
\lambda(\rho, s) = \frac{1}{2}\int_{-1}^1\log(|T'|)dx = \log(\rho)+\log(s)+1-s.
\]
The zero of $\lambda(2, s)$ is contained in $[2.67834, 2.67835]$; this interval is computed through 
the use of a rigorous Interval Newton Method \cite{TuVa}. 
\end{proof}
\begin{Remark}
An enclosure of the zero set of $\lambda(\rho, s)$ for $\rho\in [1.00781, 2]$ computed through 
the use of a rigorous Interval Newton Method is plotted 
in figure \ref{fig:zeroline}. Below the zero set the large noise limit is positive, above the
zero set is negative.	
\end{Remark}

\begin{figure}[h]
	\begin{center}
	\includegraphics[width=7cm]{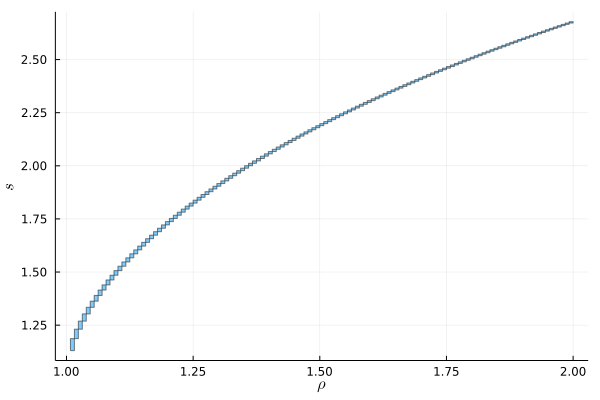}
	\end{center}
	\caption{An eclosure for the zero set of $\lambda(\rho, s)$.}\label{fig:zeroline}
\end{figure}

\begin{Remark}\label{rmk:RovellaParams}
Arguing as in \cite{GaNiPa} it is possible to show that for $C^3$ perturbations 
of the Rovella flow associated to $F$ associated to a Rovella parameter $a$,
 the Poincaré return map is such that
\begin{itemize}
\item the one dimensional map $T_a$ is such that
\[
K_2|x|^{s-1}\leq |T'_a(x)|\leq K_1 |x|^{s-1}
\]
for all $x, a$ where $s=s(a)>1$,
\item $T_a$ is $C^3$ and its derivative depends continuously on $a$
\item there exists positive constants $M_1, M_2$ independent of $a$ such that
\[
M_1|x|^r\leq |\partial_y G_a(x, y)|\leq M_2|x|^r.
\]
\end{itemize}
This allows to extend our result to perturbations of the contracting Lorenz Flow 
corresponding to a Rovella parameter. 
\end{Remark}

\section{Conclusion and outlook}\label{sec:conclusion}

Building off the previous mathematical work \cite{Ni, GaMoNi} on noise-induced order, 
this paper provides a rich class of higher-dimensional systems exhibiting noise-induced order. 
We propose here two potential avenues for future work in this direction: 

\medskip

\noindent {\bf (1) Beyond skew products. } It would be of considerable interest to provide examples of noise-induced order
	beyond the category of skew products. A natural class potentially amenable to this kind of
	analysis is Henon maps and their variants at Benedicks-Carleson-type parameters, e.g., 
	\[
	F_{a, b}(x,y) = (a + x^r + y, b x) \, , \quad b \ll 1 \, , 
	\]
	where $r > 2$ is a fixed parameter, so that $D F$ is a strong contraction along `most' of phase space, suggestive of noise-induced order. 
	
	By now there is a well-developed stochastic stability
	theory for the physical measures of such systems. However, the approach to 
	noise-induced order presented here and in \cite{Ni} requires stochastic stability
	not just of the physical measure but also of the \emph{top Lyapunov exponent} of the system.
	
\medskip 

\noindent {\bf (2) Contracting Lorenz flow. } Many contracting Lorenz models 
experience strong contraction in the vast majority of phase space, leading us to conjecture
that these models also experience noise-induced order. Deducing this from noise-induced order
for the Poincar\'e return map appears to be challenging, however. One concern is that noisy driving 
to the flow itself could destroy the Poincar\'e section with positive probability, 
making it difficult to ``lift'' results for the map to the flow. For this reason, we speculate that
it will be necessary to resort to computer-assisted methods such as those in \cite{GaMoNi} to deduce noise-induced order for contracting Lorenz flow.

\section{Appendix}

\subsection{Background on the Multiplicative Ergodic Theorem}

Below, we write $\sigma_1(A) \geq \sigma_2(A) \geq \cdots$ for the singular value of a matrix $A$, i.e., the eigenvalues of $\sqrt{A^\top A}$ counted with multiplicity. 
Recall that if $A$ is a square $d \times d$ matrix, then $\prod_i \sigma_i(A) = |\det(A)|$. 

Let $T : (X, \Fc, m) \circlearrowleft$ be an ergodic mpt of a probability space
and let $A : X \to M_{d \times d}(\R)$ be a measurable mapping.
For $x \in X, n \in \mathbb Z_{\geq 1}$, let $A^n_x = A_{T^{n-1}x} \circ \cdots \circ A_x$
and assume $\log^+\| A(x)\| = \max\{ \log\|A(x)\|, 0\}$ is in $L^1(m)$.  
By the subadditive ergodic theorem, the limits 
\[
\chi_i = \lim_n \frac1n \log \sigma_i(A^n_x)
\]
exist and are constant over $m$-typical $x \in X$ (possibly $-\infty$; here we take the convention $\log 0 = -\infty$).

Let $\lambda_1 > \lambda_2 > \cdots > \lambda_r \geq -\infty$
denote the distinct values among the $\chi_i$, and let $m_i$ denote the number of occurrences of the value $\lambda_i$ (the multiplicity of $\lambda_i$). For $x \in X, v \in \R^d$, let
\[
\lambda(x, v) = \lim_n \frac1n \log \| A^n_x (v)\| \, , 
\]
when this limit exists. 
\begin{Theorem}[Multiplicative ergodic theorem]
At $m$-a.e. $x \in X$ there is a filtration
\[
\R^d =: F_1(x) \supsetneq F_2(x) \supsetneq \cdots \supsetneq F_r(x) \supsetneq F_{r + 1}(x) := \{ 0 \}
\]
of $\R^d$ into measurably varying subspaces $F_i(x)$ with the property that for all $1 \leq i \leq r$ and for all $v \in F_i(x) \setminus F_{i + 1}(x)$, we have
\[
\lambda(x, v) = \lambda_i \,. 
\]
\end{Theorem}
Below, given a square $d \times d$ matrix $A$ and a subspace $V \subset \R^d$, 
we write $\det(A|_V)$ for the determinant of $A|_V : V \to A(V)$, using the convention
$\det(A|_V) := 0 $ if $\dim A(V) < \dim V$.
\begin{Corollary}\label{cor:subspaceGrowth}
There is a full $m$-measure set of $x \in X$ for which the following holds. For any 
$k$-dimensional subspace $V \subset \R^d$, $1 \leq k \leq d$, there are indices
$1 \leq i_1 < i_2 < \cdots < i_k \leq d$ such that
\begin{align}\label{eq:lyapBasisApp}
\lim_n \frac1n \log| \det(A^n_x|_V)| = \chi_{i_1} + \cdots + \chi_{i_k} \,. 
\end{align}
Moreover, if $V' \subsetneq V$ and $\lim_n \frac1n \log|\det(A^n_x|_{V'})| = \chi_{i_1'} + \cdots + \chi_{i_{k'}'}$, $k' = \dim V'$, then
\begin{align}\label{eq:lyapBasisApp2}
\{ i_1' , \cdots, i_{k'}'\} \subsetneq \{ i_1, \cdots, i_k \} \,. 
\end{align}
\end{Corollary}
\begin{proof}[Proof sketch]
Recall that a Lyapunov basis is a basis $\{ v_1, \cdots, v_d\}$
of $\R^d$ such that $\lambda(x, v_i) = \chi_i$ for each $1 \leq i \leq d$. 
The proof of \eqref{eq:lyapBasisApp} is to construct a \emph{Lyapunov basis} at $x$ containing
a set of $k$ vectors which span $V$. To establish \eqref{eq:lyapBasisApp2}, one constructs
 a Lyapunov basis containing $k'$ vectors spanning $V'$ and $k$ vectors spanning $V$; from 
 $V' \subsetneq V$ it is immediate that the $k'$ vectors spanning $V'$ are contained among the $k$
 vectors spanning $V$. Further details are omitted. 
\end{proof}

\subsection{Completing the proof of Corollary \ref{cor:lyapRegSkewProd}}\label{sec:appConvLE}

We present here the argument that for Leb-almost every $(x,y) \in [-1,1]^{d + 1}$, we have that
\begin{align}\label{eq:lyapRegApp}
\chi_1(\xi) = \lim_{n \to \infty} \frac1n \log || D_{(x,y)} F^n_\uo|| \quad \text{ with probability 1.}
\end{align}

\subsubsection{Preliminaries}
 We will consider an auxiliary random dynamical system $\hat F^n_\uo$ on $[-1,1]^{d + 1}$
obtained by first applying noise and then the map $F$; to wit, for $\omega \in \Omega_0$ define
$\hat F_\omega :[-1,1]^{d + 1} \circlearrowleft$ by
\[
\hat F_\omega(x,y) = F((x,y) + \omega \textrm{ mod 2}) \, , 
\]
and for $\uo = (\omega_1, \omega_2, \cdots) \in\Omega, n \geq 1$ we set $\hat F_\uo = \hat F_{\omega_n} \circ \cdots \circ \hat F_{\omega_1}$. The RDS $(\hat F_\uo^n)$ gives rise to a corresponding Markov
chain $\{ (\hat x_n, \hat y_n)\}$ on $[-1,1]^{d + 1}$ defined for initial $(\hat x_0, \hat y_0) \in [-1,1]^{d + 1}$ by
\[
(\hat x_n, \hat y_n) = \hat F^n_\uo(\hat x_0, \hat y_0) 
\]
with corresponding transition kernel $\hat P((x,y), K) = \P (\hat F_\omega(x,y) \in K)$. 

The advantage of the auxiliary Markov chain is the following regularity property not enjoyed
by the original chain $(x_n, y_n)$. 
\begin{Lemma}
The kernel $\hat P$ is strong Feller, i.e., for any bounded measurable $\varphi : [-1,1]^{d + 1} \to \R$
we have that $\hat P \varphi$ is continuous. 
\end{Lemma}
\noindent This is straightforward and follows from the fact that the convolution of two $L^2$ functions is
continuous; details are omitted. 
By \cite{Seidler}, it follows that $\hat P$ is \emph{ultra Feller}, i.e., the transition kernels 
$(x,y) \mapsto \hat P((x,y), \cdot)$ vary continuously in the TV metric $\textrm{dist}_{TV}$, defined
for Borel probability measures $\mu_1, \mu_2$ by
\[
\textrm{dist}_{TV}(\mu_1, \mu_2) = \frac12 \sup_{A} | \mu_1(A) - \mu_2(A)| \,. 
\]

We will also use the following standard point-set topology fact (proof omitted): 
\begin{Lemma}
	Let $X$ be a compact metric space and let $Y$ be a metric space. 
	Let $\Phi : X \to Y$ be a continuous map. Then, $\Phi$ is uniformly continuous.
\end{Lemma}
Since $[-1,1]^{d + 1}$ is compact, it follows that $(x,y) \mapsto \hat P((x,y), \cdot)$ is uniformly continuous
in $TV$. 

\subsubsection{Proof of \eqref{eq:lyapRegApp}}
Let $\hat A \subset [-1,1]^{d + 1}$ denote the set where
\[
\chi_1(\xi) = \lim_{n \to \infty} \frac1n \log ||D_{(x,y)} \hat F^n_\uo|| \quad \text{ with probability 1} \,. 
\]
and observe the following: 
\begin{itemize}
	\item[(i)] $\hat A$ and the set of $(x,y)$ where \eqref{eq:lyapRegApp} holds differ on a zero Lebesgue measure set, hence it suffices to prove $\Leb(\hat A) = 1$; and
	\item[(ii)] to show $\Leb(\hat A) = 1$, it suffices to show that for Leb-almost every fixed initial $(\hat x_0, \hat y_0) = (x,y)$, the stopping time
\[
T_{\hat A} = \min\{ n \geq 1 :(\hat x_n, \hat y_n) \in \hat A\}
\]
is almost-surely finite. 
\end{itemize}
Items (i) and (ii) follow from the identity 
\[
 F^{n + 1}_\uo(x,y) = \omega_{n +1} + \hat F^n_\uo\circ F(x,y) 
\]
and the fact that $D_{(x,y)} F$ is nonsingular almost-everywhere. 

To prove $T_{\hat A}$ is almost-surely finite, 
observe that $\hat P((x,y) , \hat A) = 1$ for all $(x,y) \in \hat A$. Using $TV$ uniform continuity
of $(x,y) \mapsto \hat P((x,y), \cdot)$, fix $\delta > 0$ so that $\hat P((x,y), \hat A) \geq 1/2$ 
whenever $\textrm{dist}((x,y), \hat A) < \delta$. Now, given a 
fixed, Leb-typical initial $(\hat x_0, \hat y_0) = (x,y)$, 
there is some $y' \in [-1,1]^{d + 1}$ such that $(x, y') \in \hat A$ (this uses that the stationary measure $\hat \nu_\xi$ projects to a measure $\hat \mu_\xi$ on the $x$-coordinate with density $> 0$). Let $N \geq 1$ be such that $2 c^N < \delta$, where $c \in (0,1)$ is as in condition (F), and observe that for our Lebesgue-typical $(x,y)$, we have that
\[
\textrm{dist}(\hat F^n_\uo(x,y), \hat A) \leq | \hat F^n_\uo(x,y) - \hat F^n_\uo(x, y')| \leq 2 c^n < \delta
\]
for all $n \geq N$.

We now check by induction that 
\begin{align}\label{inductionapp}
\P((\hat x_{N + j}, \hat y_{N + j}) \notin \hat A, 1 \leq j \leq i) \leq 2^{-i}
\end{align}
for all $i \geq 1$. Assuming this, we immediately obtain $\P(T_{\hat A} > N + i) \leq 2^{-i}$ which 
implies $T_{\hat A} < \infty$ with probability 1. 

To prove \eqref{inductionapp}: in the case $i = 1$ we have
\[
\P((\hat x_{N + 1}, \hat y_{N + 1}) \notin \hat A) \leq \P(\hat F^{N + 1}_\uo(x,y) \notin \hat A) 
= \int \hat P^N((x,y), d(\hat x, \hat y)) \hat P((\hat x, \hat y), \hat A^c) \leq \frac12
\]
since $\hat P^n((x,y), \cdot)$ as a measure assigns full probability to the set where $\textrm{dist}((\hat x, \hat y), \hat A) < \delta$ for all $n \geq N$. 
Assuming $\P(T_{\hat A} > N + i) \leq 2^{-i}$, we now have
\begin{align*}
\P((\hat x_{N + j}, \hat y_{N + j}) \notin \hat A, 1 \leq j \leq i + 1) & = \P((\hat x_{N + i + 1} , \hat y_{N + i + 1}) \notin \hat A | (\hat x_{N + j}, \hat y_{N + j}) \notin \hat A, 1 \leq j \leq i) \\
& \times \P((\hat x_{N + j}, \hat y_{N + j}) \notin \hat A, 1 \leq j \leq i) \\
& = \P((\hat x_{N + i + 1} , \hat y_{N + i + 1}) \notin \hat A | (\hat x_{N + i}, \hat y_{N + i}) \notin \hat A)
 \times 2^{-i} \, , 
\end{align*}
on combining the induction hypothesis and the Markov property in the last line. Now, 
\begin{gather*}
\P((\hat x_{N + i + 1}, \hat y_{N + i + 1}) \notin \hat A | (\hat x_{N + i}, \hat y_{N + i}) \notin \hat A) \\
 = \frac{1}{\hat P^{N + i}((x,y), \hat A^c)} \int_{(\hat x, \hat y) \notin \hat A} \hat P^{N + i} ((x,y), d(\hat x, \hat y))  \hat P((\hat x, \hat y), \hat A^c) \\
 \leq 1/2 \, , 
\end{gather*}
 completing the proof of \eqref{inductionapp}. \hfill $\square$

\end{document}